\documentclass{article}

\usepackage[margin=3cm]{geometry}

\usepackage[T1]{fontenc}
\usepackage[affil-it]{authblk}
\usepackage[utf8]{inputenc}
\usepackage{authblk}
\usepackage{float}

\usepackage{amsthm}
\usepackage{amssymb}
\usepackage{amsmath}
\usepackage{graphicx}
\usepackage{caption}
\usepackage{subcaption}
\usepackage{algorithmic}
\usepackage{enumerate}
\usepackage{mathtools}
\usepackage{relsize}
\usepackage{paralist}
\usepackage{bm}
\usepackage{cite}

\usepackage{fancybox,fancyheadings}

\floatstyle{ruled}
\newfloat{algorithm}{tbp}{loa}
\providecommand{\algorithmname}{Algorithm}
\floatname{algorithm}{\protect\algorithmname}

\theoremstyle{plain}
\newtheorem{thm}{Theorem}
\theoremstyle{plain}
\newtheorem{lem}{Lemma}
\theoremstyle{plain}
\newtheorem{prop}{Proposition}
\theoremstyle{remark}
\newtheorem{rem}{Remark}
\theoremstyle{plain}

\theoremstyle{plain}

\theoremstyle{definition}
\newtheorem{example}{Example}
\theoremstyle{definition}
\newtheorem{problem}{Problem}
\theoremstyle{definition}
\newtheorem{ass}{Assumption}

\title{Solving systems of polynomial inequalities with algebraic geometry methods}

\author{Laura Menini\thanks{L. Menini: menini@disp.uniroma2.it }%
, Corrado Possieri\thanks{C. Possieri: possieri@ing.uniroma2.it; Corresponding author.} and %
Antonio Tornambè\thanks{A. Tornambè: tornambe@disp.uniroma2.it} %
}

\affil{Dipartimento di Ingegneria  Civile e Ingegneria Informatica, %
Università di Roma Tor Vergata, Roma, Italy}%

\date{}
\begin{document}

\maketitle

\begin{abstract}
The goal of this paper is to provide computational tools able to find a solution of a 
system of polynomial inequalities. The set of inequalities is reformulated as a system of
polynomial equations. Three different methods,
two of which taken from the literature, are proposed to compute solutions of such a system.
An example of how such procedures can be used to solve the static output feedback
stabilization problem for a linear parametrically--varying system is reported.
\end{abstract}

\section{Introduction\label{sec:intro}}

In several control problems, it is needed to guarantee 
the existence of real solutions, and, possibly, to compute one of them,
for a system of polynomial equalities or inequalities \cite{abdallah1995applications,henrion2005positive,chesi2010lmi}.
For instance, a solution of a  set of polynomial equalities and inequalities has to be found to
solve the static output feedback stabilization problem \cite{astolfi2004static}, to compute
the equilibrium points of a nonlinear system \cite{vidyasagar2002nonlinear},
to establish if a polynomial can be written as sum of squares \cite{ParilloPhd},
to study the stability of linear systems, with structured uncertainty~\cite{chesi2007robust}.

In this paper, three algorithms, which use the tools of algebraic geometry, are used to compute
solutions of a system of polynomial equations. 
Algebraic geometry tools have been already used for control problems (see, for instance,
\cite{diop1991elimination,menini2014algebraic,menini2014CDC,possieri2014polynomial}).
The first algorithm is based on the computation of a quotient--ring basis
 \cite{cox1992ideals,cox1998using}
and of the eigenvalues of some matrices characterizing such a basis.
The second algorithm is based on the Rational Univariate Representation \cite{rouillier1999solving}
of a given ideal. The third algorithm is based on the computation of a Groebner basis \cite{cox1992ideals}
of an `extended' ideal. The first two algorithms are taken from the literature, whereas the last one
is new, to the best authors' knowledge.

Even if these techniques are able to solve only systems of equalities, by using the procedure given
in \cite{anderson1977output}, it is possible to reformulate 
a set of inequalities into a set of equalities; whence, the three mentioned algorithms can be also
used to find a solution of a set of inequalities.
Moreover, thanks to the recent advantages in Computer Algebra Systems, able to
carry out complex algebraic geometry computations (as, e.g., Macaulay2 \cite{M2}), 
by using the algorithm based on the computation of a Groebner basis of an `extended' ideal,
which is the new main result,
a solution to a set
of polynomial inequalities can be obtained also when some coefficients of the polynomials are
unknown parameters. Hence, when the values of such parameters can be assumed to be known in real time,
as for Linear Parametrically--Varying (briefly, LPV) system,
 the new method proposed here allows to compute off--line most of the solution
in parametric form, leaving only small portion of the computations to be executed in real time,
for the actual values of the parameters. In Section~\ref{sec:LPVSOF}, this method is applied
to compute a parameter dependent Static Output Feedback (briefly, SOF),
which makes a LPV system asymptotically stable.

\section{Notation and preliminaries\label{sec:basicNotions}}

In this section, some notions of algebraic geometry are recalled,
following the exposition in \cite{cox1992ideals,cox1998using}.

Let $x=[\begin{array}{ccc}
x_1 & \cdots & x_n
\end{array}]^\top$. A \emph{monomial} in $x$ is a product of the form $x_1^{\alpha_1}\cdots 
x_n^{\alpha_n}$, where $\alpha_i$, for $i=1,\dots,n$, are non--negative integers
(i.e., $\alpha\in\mathbb{Z}^n_{\geq 0}$); 
a \emph{polynomial} in $x$ is a finite $\mathbb{R}$--linear combination of monomials in $x$.
Let 
$\mathbb{R}[x]$
denote the ring of all the polynomials.

Given a set of polynomials $\{p_1,\dots,p_s\}\subset\mathbb{R}[x]$, the \emph{affine variety} defined by 
$p_1,\dots,p_s$ (see, e.g., \cite{danilov1998algebraic,liu2002algebraic}) is 
\begin{equation*}
\bm{V}(p_1,\dots,p_s)=\{x\in\mathbb{R}^n:\;p_i(x)=0,\,i=1,\dots,s\},
\end{equation*}
whereas, the \emph{semi--algebraic set} defined by $p_1,\dots,p_s$ is
\begin{equation*}
\bm{W}(p_1,\dots,p_s)=\{x\in\mathbb{R}^n:\;p_i(x)\geq 0,\,i=1,\dots,s\}.
\end{equation*}

The \emph{ideal} $\langle p_1,\dots,p_s \rangle$ in $\mathbb{R}[x]$ is the 
set of all the polynomials $q\in\mathbb{R}[x]$, which can be expressed as a finite linear 
combination of $p_1,\dots,p_s$, with polynomial coefficients $h_1,\dots,h_s\in\mathbb{R}[x]$,
i.e., $q(x)=\sum_{i=1}^{s}h_i(x)p_i(x)$. Affine varieties and ideals are notions linked by the concept
of affine variety of an ideal. Let $\mathcal{I}$ be an ideal in $\mathbb{R}[x]$, the 
\emph{affine variety of the ideal $\mathcal{I}$}, denoted by $\bm{V}(\mathcal{I})$, is the set
\begin{equation*}
\bm{V}(\mathcal{I})=\{x\in\mathbb{R}^n:\;q(x)=0,\,\forall q\in\mathcal{I} \}.
\end{equation*}

A notation needed to analyze polynomials is the \emph{monomial
ordering} on $\mathbb{R}[{x}]$, denoted as
$>$, which is a relation on the set of monomials ${x^{\alpha}},\,{\alpha}\in\mathbb{Z}_{\geq0}^{n}$,
satisfying the following properties:
\begin{inparaenum}[1)]
\item if ${\alpha}>{\beta}$ and ${\gamma}\in\mathbb{Z}_{\geq0}^{n}$,
then ${\alpha}+{\gamma}>{\beta}+{\gamma}$;
\item every nonempty
subset of monomials has a smallest element under $>$.
\end{inparaenum}
The \emph{lex ordering} is a monomial ordering, denoted by $>_{l}$
and defined as:
let ${\alpha}
$ 
and ${\beta}
\in\mathbb{Z}_{\geq0}^n$ be given, then, ${\alpha}>_{l}{\beta}$ if, in
the vector difference ${\alpha}-{\beta}\in\mathbb{Z}^n$, the first
nonzero entry is positive.

The \emph{leading term of a polynomial $f({x})$},
denoted  by $\mathrm{LT}(f({x}))$,
is, for a fixed a monomial ordering, the largest monomial appearing in $f({x})$. 
For a fixed a monomial ordering, a finite subset $\mathcal{G}=\{g_{1},\dots,g_{l}\}$
of an ideal $\mathcal{I}$ is said to be a \emph{Groebner basis}
of $\mathcal{I}$ if $\langle\mathrm{LT}(g_{1}),\dots,\mathrm{LT}(g_{l})\rangle=\langle\mathrm{LT}(\mathcal{I})\rangle$,
where the \textsl{leading term of the ideal $\mathcal{I}$} is $\mathrm{LT}(\mathcal{I}):=
\{ c{x^{\alpha}}:\exists\, f\in\mathcal{I},\text{ with }\mathrm{LT}(f({x}))=c{x^{\alpha}}\}$.
The remainder of
the division of a polynomial function $f$ for the elements of a Groebner
basis $\mathcal{G}$ of  $\mathcal{I}$,
denoted by $\overline{f}^{\mathcal{G}}$,
is unique and is a finite $\mathbb{R}$--linear combination of monomials
$x^\alpha\notin\langle \mathrm{LT}(\mathcal{I}) \rangle$
(for more details see, e.g., \cite{cox1992ideals,buchberger2006bruno}). 
Moreover, it can be easily checked that, given $f,\,g\in\mathbb{R}[x]$, one has that
$\overline{f}^{\mathcal{G}}+\overline{g}^{\mathcal{G}}=\overline{f+g}^{\mathcal{G}}$
and  that $\overline{\overline{f}^{\mathcal{G}}\cdot \overline{g}^{\mathcal{G}}}^\mathcal{G}
=\overline{f \cdot g}^{\mathcal{G}}$.

Let $\mathcal{I}$ be a given ideal in $\mathbb{R}[x_1,\dots,x_n]$. The \emph{$j$-th elimination ideal of} $\mathcal{I}$
is $\mathcal{I}^j:=\mathcal{I}\cap\mathbb{R}[x_{j+1},\dots,x_n]$. Let $\mathcal{G}$ be
a Groebner basis of $\mathcal{I}$, with respect to the lex ordering, with $x_1>_lx_2>_l\dots>_lx_n$. Then, 
by the Elimination Theorem, for every $0\leq j \leq n$,
the set $\mathcal{G}_j=\mathcal{G}\cap\mathbb{R}[x_{j+1},\dots,x_n]$ is a Groebner basis of the $j$-th elimination ideal 
$\mathcal{I}^j$.

\section{Algebraic geometry algorithms for solving systems of polynomial equations\label{sec:solPol}}

In this section, three methods  to solve a system of polynomial equations,
having a finite number of solutions, are presented. 
The first algorithm is based on the computation of a quotient--ring basis
and of the eigenvalues of some matrices characterizing this basis \cite{cox1992ideals,cox1998using}. 
The second one is based on the computation of a Rational 
Univariate Representation of the solutions
of the system of polynomial equations \cite{rouillier1999solving}. 
The third one
is based on the computation of a Groebner basis of an `extended'  ideal.
The first two methods are taken from the literature, whereas, the last one is new, 
to the best authors knowledge.
Such methods are used in Section~\ref{sec:ineq} to find a solution to a system of polynomial
inequalities.

\subsection{Solution of a system of polynomial equations by using finite--dimensional quotient rings\label{sec:eigenvalues}}
In this section, the basic notions of quotient rings are recalled and an algorithm,
taken form \cite{cox1998using} and \cite{sturmfels2002solving}, to
solve systems of polynomial equations is given.

Let $\mathcal{I}\subset\mathbb{R}[x]$ be an ideal, and let $f,\,g\in\mathbb{R}[x]$. The polynomials
$f$ and $g$ are \emph{congruent modulo $\mathcal{I}$}, denoted by $f=g\,\mathrm{mod}\,\mathcal{I}$,
if $f-g\in\mathcal{I}$. The equivalence class of $f$ modulo $\mathcal{I}$, denoted by $[f]_{\mathcal{I}}$,
is defined as $[f]_{\mathcal{I}}=\{g\in\mathbb{R}[x]:\;g=f\,\mathrm{mod}\,\mathcal{I}\}$.
The \emph{quotient} of $\mathbb{R}[x]$ modulo $\mathcal{I}$, denoted by $\mathbb{R}[x]/	\mathcal{I}$
is the set of all the equivalence classes modulo $\mathcal{I}$, 
\begin{equation*}
\mathbb{R}[x]/\mathcal{I}=\{[f]_{\mathcal{I}},\,f\in\mathbb{R}[x]\}.
\end{equation*}

Let $\mathcal{G}$ be a Groebner basis of the ideal $\mathcal{I}$, according to any monomial ordering.
By the definition of the class $[f]_{\mathcal{I}}$, one has that $\overline{f}^\mathcal{G}\in[f]_{\mathcal{I}}$.
Hence, the remainder $\overline{f}^\mathcal{G}$ can be used as a standard representative of the class
$[f]_{\mathcal{I}}$ (in the rest of this paper, the remainder $\overline{f}^\mathcal{G}$
 is identified with its class $[f]_{\mathcal{I}}$). 
Therefore, since the operations of sum and product by a constant on $\mathbb{R}[x]/\mathcal{I}$ 
have a one--to--one correspondence with the same operations on the remainders, the
elements in $\mathbb{R}[x]/\mathcal{I}$ can be added and multiplied by a constant.
Thus, the quotient ring $\mathbb{R}[x]/\mathcal{I}$ has the structure of a vector field over $\mathbb{R}$
 (it is called an \emph{algebra}).
 Since all the remainders $\overline{f}^\mathcal{G}$ are 
 $\mathbb{R}$--linear combinations of monomials, none of which is in the ideal 
 $\langle \mathrm{LT}(\mathcal{I})  \rangle$, it is possible to form a monomial basis $\mathcal{B}$ of 
 the quotient ring $\mathbb{R}[x]/\mathcal{I}$ as
\[
 \mathcal{B}=\{x^\alpha:\;x^\alpha\notin\langle\mathrm{LT}(\mathcal{I})\rangle \}.
\]
 
 The following theorem gives conditions on $\mathcal{I}$, for
 the algebra $\mathfrak{A}=\mathbb{R}[x]/\mathcal{I}$ to be
 finite--dimensional.
\begin{thm}
\cite{cox1998using}
Let $\mathcal{I}$ be an ideal in $\mathbb{R}[x]$ and let $\mathcal{G}$ be a Groebner basis of
$\mathcal{I}$, according to any monomial ordering. The following conditions are equivalent:
\begin{enumerate}
\item The algebra $\mathfrak{A}=\mathbb{R}[x]/\mathcal{I}$
	 is finite--dimensional over $\mathbb{R}$.
\item The affine variety $\bm{V}(\mathcal{I})$ is a finite set.
\item For each $i\in\{1,\dots,n\}$, there exists an $m_i\geq 0$ such that
	$x_i^{m_i}=\mathrm{LT}(g)$, for some $g\in\mathcal{G}$.	
\end{enumerate}
\label{thm:finiteness}
\end{thm}

If an ideal $\mathcal{I}$ in $\mathbb{R}[x]$ is such that one of the conditions of Theorem~\ref{thm:finiteness}
hold, then $\mathcal{I}$ is called \emph{zero dimensional}. An immediate consequence of Theorem~\ref{thm:finiteness} is that,
if the ideal $\mathcal{I}$ is zero dimensional, then any basis $\mathcal{B}$ of $\mathbb{R}[x]/\mathcal{I}$ 
has a finite number of elements, which can be chosen to be all monomials.
With such a choice (assumed in the following), given a polynomial $f\in\mathbb{R}[x]$, one can use multiplication
to define a linear map $m_f^{\mathfrak{A}}$ between $\mathfrak{A}=\mathbb{R}[x]/\mathcal{I}$ and itself.
More precisely, $m_f^{\mathfrak{A}}:\mathfrak{A}\rightarrow\mathfrak{A}$
is defined as:
\begin{equation*}
m_f^{\mathfrak{A}}([g]_{\mathcal{I}})=[f]_{\mathcal{I}}\cdot [g]_{\mathcal{I}}=[f\cdot g]_{\mathcal{I}}\in\mathfrak{A},\quad\forall[g]_{\mathcal{I}}
\in\mathfrak{A}.
\end{equation*}

Since the algebra $\mathfrak{A}$ is finitely generated over $\mathbb{R}$, then the map $m_f^{\mathfrak{A}}$ can be
represented by its associated matrix $M_f^{\mathfrak{A}}$, with respect to the chosen finite--dimensional basis
$\mathcal{B}$ of $\mathfrak{A}$.

The following two propositions characterize the elements of the affine variety $\bm{V}(\mathcal{I})$ of a zero
dimensional ideal $\mathcal{I}$.

\begin{prop}
\cite{sturmfels2002solving}
Let the ideal $\mathcal{I}$ in $\mathbb{R}[x]$ be zero dimensional. 
Let $\mathcal{B}=\{b_1,\dots,b_\ell\}$ be the basis of the finite--dimensional algebra $\mathfrak{A}=\mathbb{R}[x]/\mathcal{I}$.
Let $M_{b_i}^\mathfrak{A}$ be the matrix representing the linear map $m_{b_i}^\mathfrak{A}$, with
respect to the basis $\mathcal{B}$, for $i=1,\dots,\ell$. 
Define the real symmetric matrix 
%
%
%
$T$ as:
\begin{equation*}
[T]_{j,k}=\mathrm{Tr}(M_{b_j}^\mathfrak{A}M_{b_k}^\mathfrak{A}),
\end{equation*}
where $[T]_{j,k}$ denotes the $j\times k$th entry of $T$ and $\mathrm{Tr}(\cdot)$ is
the trace operator. The number of elements (including their multiplicity) of $\bm{V}(\mathcal{I})\subset\mathbb{R}$
equals the signature of $T$, i.e. the number of positive eigenvalues of $T$ minus the number of
negative eigenvalues of $T$.
\label{prop:numberPoints}
\end{prop}

\begin{prop}
\cite{cox1998using}
Let the assumptions of Proposition~\ref{prop:numberPoints} hold. 
Let $M_{x_i}^\mathfrak{A}$ be the matrix representing the linear map $m_{x_i}^\mathfrak{A}$, with
respect to the basis $\mathcal{B}$, for $i=1,\dots,n$. 
The real eigenvalues of
the matrix $M_{x_i}^\mathfrak{A}$ are the $x_i$--coordinates of the points of
$\bm{V}(\mathcal{I})\subset\mathbb{R}^n$, for $i=1,\dots,n$.
\label{prop:coordinates}
\end{prop}

By Proposition~\ref{prop:numberPoints} and Proposition~\ref{prop:coordinates},
Algorithm~\ref{alg:solutionMatrix} is able to solve a system of polynomial equations
having a finite number of solutions.

\begin{algorithm}[htb]
\begin{algorithmic}[1]
\REQUIRE A zero dimensional ideal $\mathcal{I}$ in $\mathbb{R}[x]$.
\ENSURE The points in $\bm{V}(\mathcal{I})\subset\mathbb{R}^n$. 
\STATE Define the algebra $\mathfrak{A}=\mathbb{R}/\mathcal{I}$.
\STATE Compute a basis $\mathcal{B}=\{b_1,\dots,b_\ell \}$ of the algebra $\mathfrak{A}$.
\FOR {$i=1$ \TO $n$}
\STATE Compute the matrix $M_{x_i}^\mathfrak{A}$, representing the map .
\STATE Compute the set $\mathcal{E}_i$ of the real eigenvalues of $M_{x_i}^\mathfrak{A}$\label{step:Ei}.
\ENDFOR
\STATE Compute the matrix $T$, such that $[T]_{j,k}=\mathrm{Tr}(M_{b_j}^\mathfrak{A}M_{b_k}^\mathfrak{A})$.
\STATE Let $d$ be equal to the signature of the matrix $T$.
\STATE Let $\mathcal{S}$ be the set of the $n$--tuple $s_j=[\begin{array}{ccc}
	s_{j,1} & \cdots & s_{j,n}\end{array}]^\top$, $s_{j,i}\in\mathcal{E}_i$, for $i=1,\dots,n$ and for $j=1,\dots,d$, 
	which are such that, for all $f\in\mathcal{I}$, $f(s_j)=0$, for $j=1,\dots,d$ \label{step:sol}.
\RETURN $\mathcal{S}=\bm{V}(\mathcal{I})$.
\end{algorithmic}
\caption{Solution of polynomial equations through computation of eigenvalues.\label{alg:solutionMatrix}}
\end{algorithm}

\begin{rem}
Let $m$ be the dimension of the basis $\mathcal{B}$. The dimension of the matrices 
$M_{x_i}^\mathfrak{A}$ is $m\times m$, for all $i\in\{1,\dots,n\}$. Therefore, the sets
$\mathcal{E}_i$, defined at Step~\ref{step:Ei}, for $i=1,\dots,n$, are composed at most by $m$ elements. Thus, Step~\ref{step:sol}
of Algorithm~\ref{alg:solutionMatrix} can be carried out by evaluating the polynomials $p_1,\dots,p_s$
at most in $m^n$ iterations.
\end{rem}

\subsection{The Rational Univariate Representation\label{sec:RUR}}
In this section, the Rational Univariate Representation (briefly, RUR), which is able to
solve a system of polynomial equations is reported, following the exposition in \cite{rouillier1999solving}.

Let $\mathbb{C}$ be the field of complex number (which is the algebraic closure of $\mathbb{R}$).
Let $\mathcal{I}$ be a zero dimensional ideal in $\mathbb{R}[x]$.
Let $\bm{V}_{\mathbb{C}}(\mathcal{I}) \subset \mathbb{C}^n$ be defined as 
\begin{equation*}
\bm{V}_{\mathbb{C}}(\mathcal{I})=\{x\in\mathbb{C}^n:\;q(x)=0,\,\forall q\in\mathcal{I}  \}.
\end{equation*}
A polynomial $q\in\mathbb{R}[x]$ is called
\emph{separating with respect to $\bm{V}_{\mathbb{C}}(\mathcal{I})$} if $\forall \alpha,\,\beta \in\bm{V}_{\mathbb{C}}
(\mathcal{I})$, $\alpha\neq \beta$, then $q(\alpha) \neq q(\beta)$.

Let $\mathcal{I}$ be a zero dimensional ideal in $\mathbb{R}[x]$, let $h$ be a polynomial in $\mathbb{R}[t]$
and let $\phi:\bm{V}_{\mathbb{C}}(\mathcal{I})\rightarrow \bm{V}_{\mathbb{C}} ( \langle h \rangle ) $ be an
isomorphism represented by a polynomial $q\in\mathbb{R}[t]$, i.e.,
$\phi(\alpha)=q(\alpha)$, for any $\alpha\in\bm{V}_{\mathbb{C}}(\mathcal{I})$.
The pair $(
\phi ,\, h)$ is a \emph{Univariate Representation of $\bm{V}_{\mathbb{C}}(\mathcal{I})$} if,
for any $\alpha\in\bm{V}_{\mathbb{C}}(\mathcal{I})$, one has that $\mu(\alpha)=\mu(q(\alpha))$,
where the symbol $\mu(\cdot)$ denotes the multiplicity of the element at argument.

On the other hand, letting $\mathcal{I}$ be a zero dimensional ideal in  $\mathbb{R}[x]$, letting $f$ be any
polynomial in $\mathbb{R}[x]$, letting $\mathcal{B}$ be a monomial basis of the algebra $\mathfrak{A}=\mathbb{R}[x]
/\mathcal{I}$, letting $M_f^{\mathfrak{A}}$ be the matrix representing the linear map $m_f^{\mathfrak{A}}$
over the basis $\mathcal{B}$ and letting $\chi_f$ be the characteristic polynomial of the matrix $M_f^\mathfrak{A}$,
define, for any $\nu\in\mathbb{R}[x]$, the following polynomial
\begin{equation*}
g_f(\nu,t)=\sum\limits_{\alpha\in\bm{V}_{\mathbb{C}}(\mathcal{I})} \mu(\alpha)\nu(\alpha)+
\sum\limits_{y\neq f(\alpha),\,y\in \bm{V}_{\mathbb{C}}(\langle \chi_f \rangle)} (t-y),
\end{equation*}
the \emph{$f$--representation of $\mathcal{I}$} is the polynomial $(n+2)$--tuple
\begin{equation*}
\{\chi_f(t),\,g_f(1,t),\,g_f(x_1,t),\cdots,g_f(x_n,t)  \},
\end{equation*}
where $\chi_f\in\mathbb{R}[t]$, and, if $f$ separates $\bm{V}_{\mathbb{C}}(\mathcal{I})$, the $f$--representation of $\mathcal{I}$
is called the \emph{Rational Univariate Representation} 
\emph{of $\mathcal{I}$ associated to $f$}.
If one is able to compute the RUR of $\mathcal{I}$ associated to $f$, then, the set $\bm{V}_{\mathbb{C}}
(\mathcal{I})$ can be obtained by computing the complex solutions to
\begin{equation}
\chi_f(t)=0.
\label{eq:complexUnivariate}
\end{equation}

Letting $\mathcal{T}$ be the set of all the complex solutions to \eqref{eq:complexUnivariate} in $t$, one has that,
by Theorem 3.1 of \cite{rouillier1999solving}, 
\begin{equation*}
\bm{V}_{\mathbb{C}}(\mathcal{I})=\{[\begin{array}{ccc}
\frac{g_f(x_1,t)}{g_f(1,t)} &  \cdots & \frac{g_f(x_n,t)}{g_f(1,t)}
\end{array}]^\top,\,\forall t\in\mathcal{T} \}.
\end{equation*}

Thus, by construction, the affine variety $\bm{V}(\mathcal{I})\subset\mathbb{R}^n$
can be computed as:
$
\bm{V}(\mathcal{I})=\bm{V}_{\mathbb{C}}(\mathcal{I})\cap \mathbb{R}^n
$.
However, in \cite{rouillier1999solving}, an alternative method to compute 
$\bm{V}(\mathcal{I})$ is given. Let $\mathcal{T}_{\mathbb{R}}$ be the set
of the real solutions to \eqref{eq:complexUnivariate}, i.e. $\mathcal{T}_{\mathbb{R}}
= \mathcal{T}\cap\mathbb{R}$. One has that
\begin{equation*}
\bm{V}(\mathcal{I})=\{[\begin{array}{ccc}
\frac{g_f(x_1,t)}{g_f(1,t)} &  \cdots & \frac{g_f(x_n,t)}{g_f(1,t)}
\end{array}]^\top,\,\forall t\in\mathcal{T}_{\mathbb{R}} \}.
\end{equation*}

A test to compute the number of 
real roots, with their multiplicity, of an univariate polynomial on
$\mathbb{R}$ is given in \cite{yang1999recent}.
Alternatively, the Sturm's Test \cite{gonzalez1998sturm} can be used.

In \cite{rouillier1999solving} an algorithm is given to compute a RUR of a 
given zero dimensional ideal. Such algorithm is not reported here for space
reasons. An implementation of this algorithms in the CAS
Maple is available through the command \textsf{RationalUnivariateRepresentation}
\cite{GRBpackage,RURpackage}.

Thus, the set of real solutions of a system of polynomial system
of equations can be computed by using Algorithm~\ref{alg:solutionRUR}.

\begin{algorithm}[htb]
\begin{algorithmic}[1]
\REQUIRE A zero dimensional ideal $\mathcal{I}$ in $\mathbb{R}[x]$.
\ENSURE The points in $\bm{V}(\mathcal{I})\subset\mathbb{R}^n$. 
\STATE Compute the RUR $\{\chi_f,\,g_f(1,t),\cdots,g_f(x_n,t) \},$ of $\mathcal{I}$.
\STATE Find the number $d$ of real roots of $\chi_f$.
\STATE Compute the set $\mathcal{T}_{\mathbb{R}}$ composed by $d$ real roots of $\chi_f$.
\STATE Define $\mathcal{S}=\{[\begin{array}{ccc}
\frac{g_f(x_1,t)}{g_f(1,t)} &  \cdots & \frac{g_f(x_n,t)}{g_f(1,t)}
\end{array}]^\top,\,\forall t\in\mathcal{T}_{\mathbb{R}} \}$.
\RETURN $\mathcal{S}=\bm{V}(\mathcal{I})$.
\end{algorithmic}
\caption{Solution of polynomial equations through RUR.\label{alg:solutionRUR}}
\end{algorithm}

\begin{rem}
Note that, before using Algorithm~\ref{alg:solutionRUR}, one has to verify that
the ideal $\mathcal{I}$ is zero dimensional. In \cite{cox1998using} an efficient
method to check this property of the ideal $\mathcal{I}$ is given.
\end{rem}

\subsection{The real Polynomial Univariate Representation\label{sec:GRNsol}}

In this section, an alternative method, with respect to the ones presented in Section~\ref{sec:eigenvalues}
and in Section~\ref{sec:RUR}, is presented. Such a method is based on the 
definition of an `extended'  ideal $\mathcal{I}_t=\mathcal{I}\cup \langle h\rangle \in\mathbb{R}[x,t]$,
where $h$ is a polynomial in $x$ and $t$, and on the computation of the Groebner basis of $\mathcal{I}_t$, 
according to the lex ordering.

For a given ideal $\mathcal{J}$, let $\mathcal{J}\in\mathbb{R}[x,t]$ and let 
\begin{equation*}
\bm{V}^t(\mathcal{J})=\{(x,\,t)\in\mathbb{R}^n\times\mathbb{R}:\;
 g(x,t)=0,\,\forall g\in\mathcal{J} \},
\end{equation*}
and let the symbol $\bm{V}^t_{\mathbb{C}}(\mathcal{J})$ denote the set
\begin{equation*}
\bm{V}^t_{\mathbb{C}}(\mathcal{J})=\{(x,\,t)\in\mathbb{C}^n\times\mathbb{C}:\;
 g(x,t)=0,\,\forall g\in\mathcal{J} \},
\end{equation*}

The \emph{projection map} is defined as the map $\pi_t:\mathbb{C}^n\times \mathbb{C}\rightarrow \mathbb{C}$,
which send each couple $(\bar{x},\,\bar{t})\in\bm{V}^t_{\mathbb{C}}(\mathcal{J})$ in $\bar{t}$.

The following theorem characterizes the projection map
and can be proved by a little modification of the proof of the Closure Theorem
given in \cite{cox1992ideals}.

\begin{thm}
\label{thm:closure}
If the ideal $\mathcal{I}_t$ in $\mathbb{R}[x,t]$ is zero dimensional,
one has that, letting $\mathcal{I}_t^n=\mathcal{I}_t\cap\mathbb{R}[t]$,
\begin{equation*}
\pi_t(\bm{V}^t_{\mathbb{C}}(\mathcal{I}_t)) = \{t\in\mathbb{C}:\;g(t)=0,\,\forall g \in \mathcal{I}_t^n\}.
\end{equation*}
\end{thm}

Consider the following assumption.

\begin{ass}
Let $\mathcal{I}\subset\mathbb{R}[x]$ be zero dimensional and let $s\in\mathbb{R}[x]$ be a polynomial 
separating with respect to $\bm{V}_{\mathbb{C}}(\mathcal{I})\subset\mathbb{C}^n$.
Let $h$ be the following polynomial in $\mathbb{R}[x,t]$:
\begin{equation*}
h(x,t) = t-s(x),
\end{equation*}
where $t$ is an auxiliary variable.
\label{ass:basic}
\end{ass}

The following two lemmas characterize the ideal $\mathcal{I}_t=\mathcal{I}\cup\langle h \rangle$
and the elimination ideal $\mathcal{I}_t^n:=\mathcal{I}_t\cap\mathbb{R}[t]$.

\begin{lem}
Let Assumption~\ref{ass:basic} hold.
Let $\mathcal{C}=\{c_1,\dots,c_\ell\}$ be any basis of the ideal $\mathcal{I}$. The ideal
$\mathcal{I}_t=\langle c_1,\dots,c_\ell,h \rangle\subset\mathbb{R}[x,t]$
is zero dimensional.
\label{lem:zeroDim}
\end{lem}

\begin{proof}
Consider the ideal $\mathcal{I}_t$. By 
\cite{cox1992ideals}, one has that 
$\mathcal{I}_t=\mathcal{I}\cup\langle h \rangle$. Hence, since
$\bm{V}^t_{\mathbb{C}}(\mathcal{I}_t)=\bm{V}^t_{\mathbb{C}}(\mathcal{I})\cap
\bm{V}^t_{\mathbb{C}}( \langle h \rangle )$ \cite{cox1992ideals} and since $s$ is separating
with respect to $\bm{V}_{\mathbb{C}}(\mathcal{I})$, one has that 
$\bm{V}^t_{\mathbb{C}}(\mathcal{I}_t)$ is a finite set. 
Thus, by Theorem~\ref{thm:finiteness}, $\mathcal{I}_t$ is zero dimensional.
\end{proof}

\begin{lem}
\label{lem:polynomialGRB}
Let Assumption~\ref{ass:basic} hold. Let $\mathcal{G}_t$ be a Groebner basis of
the ideal $\mathcal{I}_t=\mathcal{I}\cup\langle h \rangle$, with respect to any lex ordering, 
with $x_i>_lt$, for $i=1,\dots,n$.
There exists a polynomial $\eta\in\mathcal{G}_t\cap\mathbb{R}[t]$ different from the zero polynomial,
and the roots in $t$ of the polynomial $\eta$ are in 
\begin{equation*}
\mathcal{T}:=\{t\in\mathbb{C}:\;t=f(x),\,\forall x\in\bm{V}_{\mathbb{C}}(\mathcal{I})   \}.
\end{equation*}
\end{lem}

\begin{proof}
By Lemma~\ref{lem:zeroDim}, the ideal $\mathcal{I}_t$ is zero dimensional. 
Hence, by Theorem~\ref{thm:finiteness}, one has that there exists a polynomial $\eta$
different from the zero polynomial, such that $\eta\in\mathcal{G}_t\cap\mathbb{R}[t]$,
and, by the Elimination Theorem \cite{cox1992ideals}, one has that $\eta$ is a Groebner
basis of $\mathcal{I}_t^n:=\mathcal{I}_t\cap\mathbb{R}[t]$. 
Thus, by considering that $\pi_t(\bm{V}^t_{\mathbb{C}}(\mathcal{I}_t))=
\pi_t(\bm{V}^t_{\mathbb{C}}(\mathcal{I})\cap\bm{V}^t_{\mathbb{C}}( \langle h \rangle ))=
\mathcal{T}$, by Theorem~\ref{thm:closure}, one has that
$\mathcal{T} = \{t\in\mathbb{C}:\;g(t)=0,\,\forall g \in\langle \eta \rangle \}$.
\end{proof}

Let the lex ordering, with $x_1>_lx_2>_l\cdots>_lx_n>_lt$, be fixed.
A \emph{real Polynomial Univariate Representation} (briefly, \emph{PUR})
\emph{of the ideal $\mathcal{I}$} is
a $(n+1)$--tuple 
\begin{equation*}
\{\eta(t),\,g_n(x_n,t),g_{n-1}(x_{n-1},t),\dots,g_2(x_2,t),\,g_1(x_1,t)     \},
\end{equation*}
which is such that $\eta\in\mathbb{R}[t]$, $g_i\in\mathbb{R}[x_i,t]$,
$\mathrm{LT}(g_i)=x_i$, for $i=1,\dots,n$, and, letting $\mathcal{T}_\mathbb{R}$
be the set of all the real solutions to $\eta(t)=0$ in $t$, 
\begin{equation*}
\bm{V}(\mathcal{I})=\{
x\in\mathbb{R}^n:\;g_i(x_i,t)=0,\,\forall t\in\mathcal{T}_{\mathbb{R}},\,i=1,\dots,n
\}.
\end{equation*}

The following theorem gives
a constructive method to compute the real PUR of a given ideal $\mathcal{I}$.

\begin{thm}
Let Assumption~\ref{ass:basic} hold. 
Let the lex ordering, with $x_1>_lx_2>_l\cdots>_lx_n>_lt$, be fixed.
The Groebner basis $\mathcal{G}_t$
of the ideal $\mathcal{I}_t=\mathcal{I}\cup\langle h \rangle$ is a real  PUR
of $\mathcal{I}$. 
\label{thm:PUR}
\end{thm}

\begin{proof}
By Lemma~\ref{lem:polynomialGRB}, one has that there exists an $\eta
\in\mathcal{G}_t\cap\mathbb{R}[t]$ different from the zero polynomial, whose root
are in the set $\mathcal{T}=\{t\in\mathbb{C}:\;t=f(x),\,\forall x\in\bm{V}_{\mathbb{C}}(\mathcal{I})   \}$.
Let $\mathcal{T}_{\mathbb{R}}=\mathcal{T}\cap\mathbb{R}$ be the set of the real roots of $\eta$.
Let $\pi_{x_i}^{\mathbb{R}}$ denote the \emph{projection map}, $\pi_{x_i}^{\mathbb{R}}:\mathbb{R}^n\times \mathbb{R}\rightarrow
\mathbb{R}$, which maps each couple $(\bar{x},\,\bar{t})$ in $\bar{x}_i$, for $i=1,\dots,n$.
By the Lagrange Interpolation Formula \cite{sauer1995multivariate}, there exists a polynomial
$\varrho_i(t)\in\mathbb{R}[t]$, which is such that $\pi_{x_i}^\mathbb{R}((\bar{x},\bar{t}))=
\varrho_i(\bar{t})$, for all the couples $(\bar{x},\bar{t})\in\bm{V}^t(\mathcal{I}_t)$, for $i=1,\dots,n$. 
Hence, in the ideal $\mathcal{I}_t$
there exists polynomials $w_i=x_i-\varrho_i(t)$, for $i=1,\dots,n$, and, by the definition of a Groebner basis,
one has that a Groebner basis of $\mathcal{I}_t$ is $\{\eta,w_n,w_{n-1},\dots,w_1 \}$ and, 
by construction, this is a real PUR of the ideal $\mathcal{I}$.
\end{proof}

By Theorem~\ref{thm:PUR}, Algorithm~\ref{alg:solutionPUR} is able to compute the real
solutions to a system of equations.

\begin{algorithm}[htb]
\begin{algorithmic}[1]
\REQUIRE A zero dimensional ideal $\mathcal{I}$ in $\mathbb{R}[x]$.
\ENSURE The points in $\bm{V}(\mathcal{I})\subset\mathbb{R}^n$. 
\STATE Define a random polynomial $s(x)\in\mathbb{R}[x]$ and  the 
	polynomial $h(x,t)=t-s(x)$\label{step:randPol}.
\STATE Letting $\{c_1,\dots,c_\ell\}$ be a set of generators of the ideal $\mathcal{I}$,
	define the ideal $\mathcal{I}_t=\langle c_1,\dots,c_\ell,h \rangle$.
\STATE Compute a Groebner basis $\mathcal{G}_t$ of the ideal $\mathcal{I}_t$,
	according to lex ordering, with $x_1>_lx_2>_l\cdots>_lx_n>_lt$.
\STATE Verify that $\mathcal{G}_t$ is a real PUR, otherwise return to step~\ref{step:randPol}.
\STATE Let $\mathcal{G}_t=\{\eta(t),\,g_n(x_n,t),\dots,\,g_1(x_1,t) \}$, where
	$g_i=x_i-\varrho_i(t)$, with $\varrho_i\in\mathbb{R}[t]$, for $i=1,\dots,n$.
\STATE Find the number $d$ of real solutions of $\eta(t)=0$.
\STATE \label{step:solpol}Compute the set $\mathcal{T}_{\mathbb{R}}$ of the $d$ real solutions of $\eta(t)=0$.
\STATE Define $\mathcal{S}=\{[\begin{array}{ccc}
\varrho_1(t) & \cdots & \varrho_n(t)
\end{array}]^\top,\;\forall t \in\mathcal{T}_{\mathbb{R}}  \}$.
\RETURN $\mathcal{S}=\bm{V}(\mathcal{I})$.
\end{algorithmic}
\caption{Solution of polynomial equations through PUR.\label{alg:solutionPUR}}
\end{algorithm}

\begin{rem}
It can be easily proved that, if one define a random polynomial $s(x)$, there is a probability of 1
that it is separating (i.e., there exist isolated monomial coefficients which makes the polynomial
$s$ be not separating). Hence, the iterations required by Algorithm~\ref{alg:solutionPUR}
are finite.
\end{rem}

\begin{rem}
By \cite{alonso1996zeros}, one has that the numerical computation of roots of the polynomial $\eta$, obtained by using 
Algorithm~\ref{alg:solutionPUR} is, generally, numerically more complex than the computation of the roots of
the polynomial $\chi_f$, obtained by using Algorithm~\ref{alg:solutionRUR}.
However, since the representation of $\bm{V}(\mathcal{I})$ obtained by using Algorithm~\ref{alg:solutionPUR}
is polynomial, 
it can be preferred to the rational one obtained by using Algorithm~\ref{alg:solutionRUR}.
\end{rem}

\begin{rem}
\label{rem:param}
Note that the computations required by Algorithm~\ref{alg:solutionMatrix} and Algorithm~\ref{alg:solutionPUR} 
can be carried out also when some (or, possibly, all the) coefficients of the polynomials $p_1,\dots,p_s$ are functions of 
some parameters.
However, even if the computation of the matrices $M_{x_i}^\mathfrak{A}$ is generally faster then
the computation of the Groebner basis $\mathcal{G}_t$, since the computations which have to be carried out
at Step~\ref{step:sol} of Algorithm~\ref{alg:solutionMatrix} can be very computationally expensive, in many 
cases of practical interest, Algorithm~\ref{alg:solutionPUR} may be preferred, because the
computations needed to solve Step~\ref{step:solpol} of Algorithm~\ref{alg:solutionPUR} 
can be carried out, generally, in a faster way.
\end{rem}

\section{Solution of systems of polynomial inequalities \label{sec:ineq}}

In this section, a procedure to compute, if any,  a solution of a system of polynomial
inequalities is given. This method, based on penalizing variables and reported, e.g., in \cite{anderson1977output},
is connected to the solution of a set 
of polynomial equations, which can be computed with the algorithms
of Section~\ref{sec:solPol}.

Consider the following problem.
\begin{problem}
Let $x=[\begin{array}{ccc}
x_1 & \cdots & x_n
\end{array}]^\top$.
Let the set of polynomials $\mathcal{P}=\{p_1,\dots,p_s\}\subset\mathbb{R}[x]$ be given. 
\begin{enumerate}[(a)]
\item Find, if any, a point $\bar{x}\in\bm{W}(p_1,\dots,p_s)$.\label{prob:a}
\item Let $\{p_1,\dots,p_\ell\}\subseteq\mathcal{P}$. Find, if any, a point 
	$\bar{x}\in\bm{W}(p_1,\dots,p_s)$, $\bar{x}\notin\bm{V}(p_1,\dots,p_\ell)$.\label{prob:b}
\end{enumerate}
\label{prob:dis}
\end{problem}

Note that a solution $\bar{x}$ to Problem~\ref{prob:dis}~\eqref{prob:a} is a solution of
the system of polynomial inequalities $p_i(x)\geq 0$, for $i=1,\dots,s$, whereas, a
solution $\bar{x}$ to  Problem~\ref{prob:dis}~\eqref{prob:b} is a solution of the system of
polynomial inequalities $p_i(x)>0$, for $i=1,\dots,\ell$, and $p_i(x)\geq 0$,
for $i=\ell+1,\dots,s$.

Let $v=[\begin{array}{ccc}
	v_1 & \cdots & v_s
	\end{array}]^\top\in\mathbb{R}^s$ and let
$w=[\begin{array}{ccc}
w_1 & \cdots & w_s
\end{array}]^\top\in\mathbb{R}^s$ be auxiliary variables.
By \cite{anderson1977output}, Problem~\ref{prob:dis}~\eqref{prob:a} can be solved 
with the following procedure:
\begin{enumerate}
\item Let $\alpha_i,\,\beta_i$, for $i=1,\dots,n$, and $\gamma_k,\,\delta_k$, for $k=1,\dots,s$,
	be fixed random real numbers.
\item Define the polynomial function $J\in\mathbb{R}[x,w]$:
	\begin{equation}
	J=\sum_{i=1}^{n} \alpha_i(x_i-\beta_i)^2+\sum_{k=1}^s\gamma_k(w_k-\delta_k)^2.
	\label{eq:J}
	\end{equation}
\item Define the polynomial function $H\in\mathbb{R}[x,v,w]$\label{step:Hdef}
	\begin{equation*}
	 H=J+\sum_{k=1}^{s}v_k(p_k-w_k^2).
	\end{equation*}
\item Solve the following
	polynomial system of equations \label{step:solve}
	\begin{subequations}
	\begin{eqnarray}
	\frac{\partial H(x,v,w)}{\partial x} & = & 0,\\
	\frac{\partial H(x,v,w)}{\partial v} & = & 0,\\
	\frac{\partial H(x,v,w)}{\partial w} & = & 0,
	\end{eqnarray}
	\label{eq:Astat}%
	\end{subequations}
	in $[\begin{array}{ccc}
	x^\top & v^\top & w^\top
	\end{array}]^\top\in\mathbb{R}^{n+2s}$.
	
\item Let $\pi_x:\mathbb{R}^{n+2s}\rightarrow\mathbb{R}^n$ be the map which maps each
	vector $[\begin{array}{ccc}
	x^\top & v^\top & w^\top
	\end{array}]^\top\in\mathbb{R}^{n+2s}$ in $x\in\mathbb{R}^n$ and let $\mathcal{S}$ be
	the set of the solutions to \eqref{eq:Astat}. A solution to 
	Problem~\ref{prob:dis}~\eqref{prob:a} is given by
	\begin{equation*}
	\{x\in\mathbb{R}^n:\;x=\pi_x(\zeta),\,\forall\zeta\in\mathcal{S}   \}.
	\end{equation*}
\end{enumerate}

On the other hand, always by \cite{anderson1977output}, a solution to Problem~\ref{prob:dis}~\eqref{prob:b}
can be obtained by using the same  procedure used to solve Problem~\ref{prob:dis}~\eqref{prob:a}, 
by changing only step~\ref{step:Hdef}):
\begin{enumerate}
\setcounter{enumi}{2}
\item Define the polynomial function $H\in\mathbb{R}[x,v,w]$
	\begin{equation*}
	 H=J+\sum_{k=1}^{\ell}v_k(w_k^2p_k-1)+\sum_{k=\ell+1}^{s}v_k(p_k-w_k^2).
	\end{equation*}
\end{enumerate}

\begin{rem}
\label{rem:nosolutions}
By \cite{anderson1977output}, one has that there exists a solution to \eqref{eq:Astat}
if and only if there exists a solution to Problem~\ref{prob:dis}.
\end{rem}

\begin{rem}
Note that the solution to Problem~\ref{prob:dis}~\eqref{prob:a} (respectively, Problem~\ref{prob:dis}~\eqref{prob:b})
obtained by using the procedure given in \cite{anderson1977output} corresponds to computing the 
stationary points of  the function $J$
in \eqref{eq:J}, subject to the constraints $p_i(x)=w_i^2$, for $i=1,\dots,s$ 
(respectively, $w_i^2p_i(x)=1$, for $i=1,\dots,\ell$, and $p_j(x)= w_j^2$, for $j=\ell+1,\dots,s$), 
where $w_i^2>0$, for $i=1,\dots,s$. Therefore, since a point
$[\begin{array}{ccc}
\bar{x}^\top & \bar{v}^\top & \bar{w}^\top
\end{array}]^\top \in\mathcal{S}$ is such that $p_i(\bar{x})=\bar{w}_i^2\geq 0$
(respectively, $\bar{w}_i^2p_i(\bar{x})=1$, for $i=1,\dots,\ell$, and $p_j(x)= w_j^2$, for $j=\ell+1,\dots,s$),
one has that $\pi_x([\begin{array}{ccc}
\bar{x}^\top & \bar{v}^\top & \bar{w}^\top
\end{array}])=\bar{x}$ is a solution to Problem~\ref{prob:dis}~\eqref{prob:a} (respectively, Problem~\ref{prob:dis}~\eqref{prob:b}).
\end{rem}

By \cite{anderson1977output} one has that the ideal $\langle \frac{\partial H(x,v,w)}{\partial x},
\frac{\partial H(x,v,w)}{\partial v},
\frac{\partial H(x,v,w)}{\partial w}   \rangle$ is, for almost any choice of 
$\alpha_i,\,\beta_i$,  $\gamma_k,\,\delta_k$, zero dimensional,
Hence, step \ref{step:solve}) of this procedure can be actually carried out by using one of the three procedures
given in Algorithm~\ref{alg:solutionMatrix}, Algorithm~\ref{alg:solutionRUR} or Algorithm~\ref{alg:solutionPUR}
in Section~\ref{sec:solPol}. 

\begin{rem}
Note that, as it is pointed out in Remark~\ref{rem:param}, Algorithm~\ref{alg:solutionMatrix} and
Algorithm~\ref{alg:solutionPUR} can be used also when the coefficients of the polynomials depend on 
some parameters. Hence, Problem~\ref{prob:dis} can be solved also with parametric coefficients
of the polynomials $p_1,\dots,p_s$ . 
Algorithm~\ref{alg:solutionPUR} may be preferred to Algorithm~\ref{alg:solutionMatrix},
because the computations needed to carry out Step~\ref{step:solpol} of the first one may
be faster then the ones needed to carry out Step~\ref{step:sol} of the latter one.
\end{rem}

 
 \begin{example}
 \label{example:inters}
 Let $n=2$, $x=[\begin{array}{cc}
 x_1 & x_2
 \end{array}]^\top$, $s=2$, 
 $p_1(x)= -(16-x_1^2)x_2^2+(-16+x_1^2+8x_2)^2$,
 $p_2(x)=5x_1^2-x_1^4-4x_2^2+x_2^4$.
 Let $v=[\begin{array}{cc}
  v_1 & v_2
  \end{array}]^\top$, $w=[\begin{array}{cc}
   w_1 & w_2
   \end{array}]^\top$.
 The goal is to compute a point $\bar{x}\in\bm{W}(p_1,p_2)\subset\mathbb{R}^2$,

 By using the procedure given in Section~\ref{sec:ineq} one has to solve the following 
 set of polynomial equations:
 \begin{equation}
 q_j(x,v,w)=0,\quad \text{for }j=1,\dots,6,\label{eq:systex1}%
 \end{equation}
 where $q_1 = v_1 (10 x_1-4 x_1^3)+2 v_2 x_1 ((w_2^2-x_2){}^2+2 (-8
    w_2^2+x_1^2+8 (x_2-2)))+2 \alpha _1 (x_1-\beta _1)$, 
$q_2=4 v_1 ((w_1^2-x_2){}^2-2) (x_2-w_1^2)-2 v_2
   ((x_1^2+48) w_2^2-48 x_2-x_1^2 (x_2+8)+128)+2 \alpha
   _2 (x_2-\beta _2)$,
$q_3 =  -x_1^4+5 x_1^2+(w_1^2-x_2)^4-4 (w_1^2-x_2)^2$,
$q_4= (-8 w_2^2+x_1^2+8 (x_2-2))^2+(x_1^2-16)
   (w_2^2-x_2)^2$,
$q_5 = 8 v_1 w_1 (2-(w_1^2-x_2){}^2) (x_2-w_1^2)+2 \gamma _1
   (w_1-\delta _1)$,
$q_6 =  4 v_2 w_2 ((x_1^2+48) w_2^2-48 x_2-x_1^2
   (x_2+8)+128)+2 \gamma _2 (w_2-\delta _2) $ and 
$\alpha_1,\,\alpha_2,\,\beta_1,\,\beta_2,\,\gamma_1,\,\gamma_2,\,
\delta_1,\,\delta_2$ are random values.
  
 \begin{figure}[htb]
         \centering
 %
 %
 %
 %
         \includegraphics[width=0.5\textwidth]{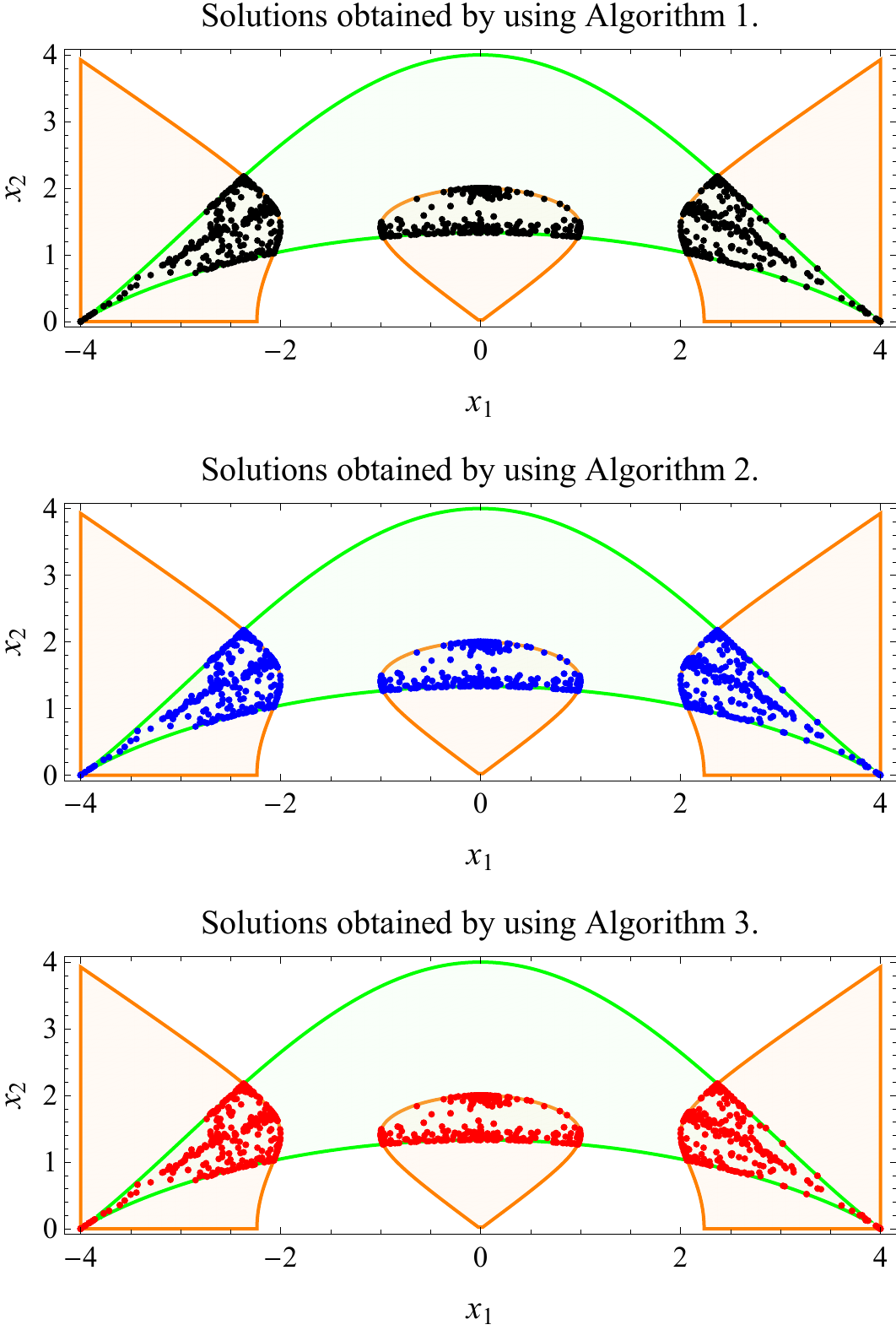}
                 
         \caption{Solutions to the problem of Example~\ref{example:inters}, for 100 different values of the random constants.\label{fig:exampletoy}}
 \end{figure}
 
 The three algorithms given in Section~\ref{sec:solPol} have been used to
  solve such a problem, with the same random choices of the coefficients $\alpha_1,\,\alpha_2,\,\beta_1,\,
  \beta_2,\,\gamma_1,\,\gamma_2,\,\delta_1\text{ and }\delta_2$.
  Figure~\ref{fig:exampletoy} shows the results obtained by using these algorithms, for 100
  different choices of the random parameters $\alpha_1,\,\alpha_2,\,\beta_1,\,
   \beta_2,\,\gamma_1,\,\gamma_2,\,\delta_1\text{ and }\delta_2$.
   As such a figure shows, these procedures are able to
   solve a set of inequalities. 
 \end{example}

\section{Application to a LPV SOF stabilization problem\label{sec:LPVSOF}}

In this section, the techniques given in this paper are used to solve
the Static Output Feedback problem \cite{syrmos1997static}
for a Linear Parametrically--Varying
system \cite{mohammadpour2012control}. 

Consider the following missile model \cite{nichols1993gain,scherer1997parametrically}
\begin{subequations}
\begin{eqnarray}
\dot{\alpha} & = & \kappa_\alpha M((a_n\alpha^2+b_n\alpha+c_n) \alpha + d_n\delta)+q,\\
\dot{q} & = & \kappa_q M((a_m\alpha^2+b_m\alpha+c_m) \alpha + d_m\delta),
\end{eqnarray}
\label{eq:LPV}%
\end{subequations}
where $\alpha$ is the angle of attack, $q$ is the pitch rate, $M$ is the Mach
number of the missile, $\kappa_\alpha$, $\kappa_q$, $a_n$, $b_n$, $c_n$, $d_n$, $a_m$, $b_m$, $c_m$, $d_m$
are known aerodynamic coefficients and $\delta$ is the tail fin deflection, which is
considered as the control input. It is assumed that the only available measure is the
 angle $\alpha$.

 System~\eqref{eq:LPV} can be rewritten as the following LPV system:
\begin{subequations}
\begin{eqnarray}
\dot{\alpha} & = & \theta_1 \alpha +q +\kappa_\alpha M d_n\delta\\
\dot{q} & = & \theta_2 \alpha +\kappa_q M d_m\delta,
\end{eqnarray}
\label{eq:LPVr}%
\end{subequations}
where $\theta_1  =   \kappa_\alpha M(a_n\alpha^2+b_n\alpha+c_n)$ and
$\theta_2  =   \kappa_q M(a_m\alpha^2+b_m\alpha+c_m)$.
Hence, letting $\theta = [\begin{array}{cc}
\theta_1 & \theta_2
\end{array}]^\top$, 
\begin{equation*}
A(\theta) = \left[\begin{array}{cc}
\theta_1 & 1 \\
\theta_2 & 0
\end{array}\right],\;
B = \left[\begin{array}{c}
\kappa_\alpha M d_n\\
\kappa_q M d_m
\end{array}\right],\;
C = [\begin{array}{cc}
1 & 0
\end{array}],
\end{equation*}
$x=[\begin{array}{cc}
\alpha & q
\end{array}]^\top$, 
system~\eqref{eq:LPVr} can be written as
\begin{subequations}
\label{eq:linearized}%
\begin{eqnarray}
\dot{x} & = & A(\theta)x + B u,\\
y & = & Cx.
\end{eqnarray}
\end{subequations}

\begin{problem}
\label{prob:SOFLPV}
Let system~\eqref{eq:linearized} be given and let $\delta = K(\theta)\alpha$, where $K(\theta)$,
is a scalar gain, dependent on the vector $\theta$. Find a $K(\theta)$ and 
$\mathcal{L}\subset\mathbb{R}^2$, 
such that the closed loop system
\begin{equation}
\dot{x} = (A(\theta)+B K(\theta) C) x
\label{eq:closedloop}
\end{equation} 
is exponentially stable, with attraction domain containing $\mathcal{L}$.
\end{problem}

\begin{rem}
Consider that, by construction, the parameters $\theta_1$ and $\theta_2$ are dependent
only on the state $\alpha$. Hence, if  Problem~\ref{prob:SOFLPV} can be solved,
system~\eqref{eq:linearized} can be stabilized by measuring $\alpha$.
\end{rem}

Consider the closed loop dynamic matrix
\begin{equation*}
\tilde{A}(\theta) = A(\theta)+B K(\theta) C. 
\end{equation*}

Let $p_{\tilde{A}}$ be the characteristic polynomial of the matrix $\tilde{A}(\theta)$,
\begin{equation*}
p_{\tilde{A}}(s,\theta)=s^2+p_1(\theta)s+p_2(\theta),
\end{equation*}
where $p_1(\theta)=-K M \kappa _{\alpha } d_n-\theta_1$ and $p_2(\theta)=-K M d_m \kappa _q-\theta_2$.
Let $\Theta$ be a subset of $\mathbb{R}^2$.
One has that, for each fixed ${\theta}\in\Theta$, the eigenvalues of
$A({\theta})$ are complex conjugate if 
\begin{subequations}
\begin{equation}
p_1^2({\theta})-4p_2({\theta})<0,
\end{equation}
and the real parts of such eigenvalues are lower than $-\lambda$ if 
\begin{equation}
p_1({\theta})>2\lambda.\label{eq:eigs}
\end{equation}
\label{eq:conditions}%
\end{subequations}
where $\lambda$ is a fixed real value greater than zero.

By considering that \eqref{eq:conditions} are polynomial inequalities, parametrized in the unknown
vector $\theta$, one has that they can be actually solved by using the procedure given in Section~\ref{sec:ineq}
to transform such a set of inequalities into a set of equalities, and Algorithm~\ref{alg:solutionPUR} 
to compute a solution $K$, depending on $\theta$, of the set of equalities.
By using such a method one obtains the following two polynomials:
\begin{equation*}
\eta(\theta,t), \quad \varrho_K(\theta,t).
\end{equation*}

For each fixed value $\bar{\theta}$, one can obtain the correspondent gain $K(\bar{\theta})$,
which is such that \eqref{eq:conditions} holds, by computing a solution $\bar{t}$ of the equation
$\eta(\bar{\theta},t)=0$ and
setting $K = \varrho_K(\bar{\theta},\bar{t})$. Note that, by Remark~\ref{rem:nosolutions},
if one is not able to find a solution to $\eta(\bar{\theta},t)=0$, then there exists no solution to
\eqref{eq:conditions} for such a value of $\bar{\theta}$ and for such 
a $\lambda$.

In the rest of this section, Assumption~\ref{ass:KKK} is made.
\begin{ass}
\label{ass:KKK}
Let $K(\theta)=\varrho_K(\theta,\bar{t}_\theta)$, where $\bar{t}_\theta$ is such that
\begin{equation*}
\eta(\theta,\bar{t}_\theta)=0.
\end{equation*}
\end{ass}

Consider now system~\eqref{eq:LPV}, with $\delta = K(\theta)\alpha$.
One has that
\begin{subequations}
\begin{eqnarray}
\dot{\theta}_1 & = & \kappa_\alpha M(2a_n\alpha+b_n)\dot{\alpha}\\
\dot{\theta}_2 & = & \kappa_q M(2a_m\alpha+b_m)\dot{\alpha}
\end{eqnarray}
\label{eq:dottheta}
\end{subequations}
where $\dot{\alpha}=\kappa_\alpha M((a_n\alpha^2+b_n\alpha+c_n) \alpha + d_nK(\alpha)\alpha)+q$. Hence, by the definition of the parameters $\theta_1$
and $\theta_2$ and of their time derivatives, one has that there exist polynomial
functions $\phi_1,\,\phi_2,\,\psi_1$ and $\psi_2\in\mathbb{R}[x]$,
which are such that
\begin{equation*}
\begin{array}{rclcrcl}
\theta_1 & = &  \phi_1(x), & \quad &
\theta_2 & = & \phi_2(x),\\
\dot{\theta}_1 & = & \psi_1(x), & &
\dot{\theta}_2 & = & \psi_2(x).
\end{array}
\end{equation*}

\begin{ass}
Let system~\eqref{eq:closedloop} be given, 
let $\mathcal{D}$ be a subset of the state space of system~\eqref{eq:closedloop},
with $0\in\mathcal{D}$, and let
\begin{subequations}
\begin{eqnarray}
\!\!\!\!\!\!\!\!\!\!\mathcal{E} &\!\!\!\! = \!\!\!\! & \{\theta\in\mathbb{R}^2:\exists x \in \mathcal{D}:\theta_1=\phi_1(x),\,
\theta_2=\phi_2(x)\},\label{eq:eee}\\
\!\!\!\!\!\!\!\!\!\!\mathcal{F} &\!\!\!\! = \!\!\!\!&  \{\omega\in\mathbb{R}^2:\exists x \in \mathcal{D}:\omega_1=\psi_1(x),\,
\omega_2=\psi_2(x)\}.\label{eq:fff}
\end{eqnarray}
\end{subequations} 
\begin{enumerate}
\item 
	 \label{point:lip} The matrix $\tilde{A}(\theta)$ is Lipschitz in $\theta$ in $\mathcal{E}$, i.e. $\exists L_A$:
	\begin{equation*}
	\Vert \tilde{A}(\theta)-\tilde{A}(\theta') \Vert < L_A \Vert \theta - \theta' \Vert,\quad\text{for all }\theta,\,\theta'\in\mathcal{E}.
	\end{equation*}
\item \label{point:norm}There exist constants $m\geq 1$ and $\lambda>0$ such that, for
	each fixed $\bar{\theta}\in\mathcal{E}$, any solution of \eqref{eq:closedloop}
	is such that 
	\begin{equation*}
	\Vert x(t) \Vert \leq m e^{-\lambda t} \Vert x(0) \Vert,\quad \forall t\geq 0.
	\end{equation*}
\item \label{point:dotth}Let $L_A$ be the Lipschitz constant of item~\ref{point:lip}) and let
	$m$ and ${\lambda}$ be the constants of item~\ref{point:norm}). One has that
	\begin{equation*}
	\Vert \dot{\theta}  \Vert < \frac{{\lambda}^2}{4 L_A m \log(m)},\quad \text{for all }\dot{\theta}\in\mathcal{F}.
	\end{equation*}
\end{enumerate}
\label{ass:exp}
\end{ass}

The following three propositions show that for system~\eqref{eq:closedloop},
under Assumption~\ref{ass:KKK}, there exists a domain $\mathcal{D}$, with $0\in\mathcal{D}$,
such that items \ref{point:lip}) -- \ref{point:dotth}) of
Assumption~\ref{ass:exp} hold, for some $L_A>0$, $m>0$ and $\lambda>0$.

\begin{prop}
Let Assumption~\ref{ass:KKK} hold. One has that there exists a domain 
$\mathcal{D}_1\subset\mathbb{R}^n$,
with $0\in\mathcal{D}_1$,
 such that the closed loop 
dynamic matrix $\tilde{A}(\bar{\theta})$ is Lipschitz in $\bar{\theta}$ in $\mathcal{E}_1$,
 with $\mathcal{E}_1$ defined as in \eqref{eq:eee}, with $\mathcal{D}$ replaced by $\mathcal{D}_1$.
\label{prop:lipschi}
\end{prop}

\begin{proof}
The proof of this proposition follows directly by the definition of the matrix
${A}(\theta)$ and of the matrix $K(\theta)$. As a matter of fact, the matrix $A(\theta)$
is trivially Lipschitz, whereas, the matrix $K(\theta)$ is obtained by computing the roots
of a polynomial whose coefficients are polynomially dependent on the parameters
vector $\theta$ and using the polynomial $\varrho_K$. 
On the other hand $K(\theta)$ is
differentiable for all the values of $\theta$ for which
the following equalities in $t$
\begin{equation}
\eta(t,\theta) =  0,\quad \quad 
\frac{\partial \eta(t,\theta)}{\partial t}  =  0,
\label{eq:discriminant}%
\end{equation}
have not real solution.
As detailed in Remark~\ref{rem:critic} below, one has that
the system of equalities \eqref{eq:discriminant} has no solution
for $\theta=[\begin{array}{cc}
\phi_1(x) & \phi_2(x)
\end{array}]^\top$, $\forall x \in\mathbb{R}^2$.
Therefore,  there exists a  bounded domain $\mathcal{D}_1$, with $0\in\mathcal{D}_1$, such that \eqref{eq:discriminant} does
not have any real solution,
for all $\theta\in\mathcal{E}_1$. Hence,
also the function $K(\theta)$ is Lipschitz in any bounded domain $\mathcal{D}_1$, with $0\in\mathcal{D}_1$,
for some $L_A$. Note that choosing a larger $\mathcal{D}_1 $ will render $L_A$ larger,
in general cases.
\end{proof}

\begin{rem}
\label{rem:critic}
Note that the set $\mathcal{N}=\{x\in\mathbb{R}^2:\;\eqref{eq:discriminant}\text{ holds, with }$\linebreak$\theta=[\begin{array}{cc}
\phi_1(x) & \phi_2(x)
\end{array}]^\top\}$
can be computed by defining the ideal \linebreak$\mathcal{J}=\langle \eta(t,\theta),\frac{\partial \eta(t,\theta)}{\partial t},\theta_1-\phi_1(x),
\theta_2-\phi_2(x) \rangle$
and by computing any Groebner basis $\mathcal{G}$,
with respect to the lex order, with $t>_l \theta_1 >_l \theta_2>_l q>_l \alpha$, of such an ideal. Hence, letting $g_1,\dots,g_l$ be the
polynomials in $\mathbb{R}[x]$ such that $\{g_1,\dots,g_l\}=\mathcal{G}\cap\mathbb{R}[x]$,
all the points in the set $\mathcal{N}$ can be obtained as the solution to the following equalities 
\begin{equation*}
g_1(x)  =  0,\quad
\cdots\quad 
g_l(x)  =  0,
\end{equation*}
which can be solved with the algorithms of Section~\ref{sec:solPol}.
By applying Algorithm~\ref{alg:solutionPUR}, it has be proved that $\mathcal{N}$ is empty.
\end{rem}

The following proposition can be easily proved by considering that, by using the techniques
given in \cite{yang1999recent}, there exist a domain 
$\mathcal{D}_2\in\mathbb{R}^2$, with $0\in\mathcal{D}_2$,
such that $\eta(\theta,t)=0$ has a solution in $t$
for $\theta=[\begin{array}{cc}
\phi_1(x) & \phi_2(x)
\end{array}]^\top$, $\forall x\in\mathcal{D}_2$ and that,
by  Assumption~\ref{ass:KKK}, one has that \eqref{eq:conditions} holds for all  $x\in\mathcal{D}_2$.

\begin{prop}
\label{prop:boundeig}
Let Assumption~\ref{ass:KKK} hold. There exists a domain 
$\mathcal{D}_2\subset\mathbb{R}^n$,
with $0\in\mathcal{D}_2$,
 such that, letting $\mathcal{E}_2$
be defined as in \eqref{eq:eee}, with $\mathcal{D}$ replaced by $\mathcal{D}_2$, there exists constant $m\geq 1$ and
${\lambda}>0$, such that,
any solution of \eqref{eq:closedloop} is such that 
\begin{equation}
\Vert x(t) \Vert \leq m e^{-{\lambda} t} \Vert x(0) \Vert,\quad \text{for each fixed }\bar{\theta}\in\mathcal{E}_2.
\label{eq:exponent}
\end{equation}
\end{prop}


\begin{prop}
\label{prop:boundedtheta}
Let Assumption~\ref{ass:KKK} hold. 
Let $L_A$ be the Lipschitz constant chosen in the proof of Proposition~\ref{prop:lipschi}
and let $m$ and ${\lambda}$ be the constants chosen in the proof of Proposition~\ref{prop:boundeig}.
There exists $\mathcal{D}_3\subset\mathbb{R}^2$,
with $0\in\mathcal{D}_3$,
such that
\begin{equation}
\Vert \dot{\theta}  \Vert < \frac{{\lambda}^2}{4 L_A m \log(m)},\quad \text{for all }\dot{\theta}\in\mathcal{F},
\label{eq:boundedot}
\end{equation}
 where $\mathcal{F}$ is defined as in \eqref{eq:fff}, 
with $\mathcal{D}$ replaced by $\mathcal{D}_3$.
\end{prop}

\begin{proof}
By \eqref{eq:dottheta}, $\dot{\theta}$ is a linear function of $\dot{\alpha}$, and $\frac{\dot{\theta}_i}{\dot{\alpha}}$
is affine in $\alpha$. Hence, \eqref{eq:boundedot} holds in some domain $\mathcal{D}_3$, with $0\in\mathcal{D}_3$.
\end{proof}

By Proposition~\ref{prop:lipschi}, Proposition~\ref{prop:boundeig} and Proposition~\ref{prop:boundedtheta},
one has that system~\eqref{eq:closedloop} is such that Assumption~\ref{ass:exp} holds, with 
$\mathcal{D}=\mathcal{D}_1\cap\mathcal{D}_2\cap\mathcal{D}_3$, where $\mathcal{D}_1$, $\mathcal{D}_2$ and
$\mathcal{D}_3$ are the domains chosen in the proofs of Proposition~\ref{prop:lipschi}, Proposition~\ref{prop:boundeig} and Proposition~\ref{prop:boundedtheta}, respectively.
The following two theorems and lemma show that the gain $K(\theta)$
and a domain $\mathcal{L}\subset\mathcal{D}$
are a solution to Problem~\ref{prob:SOFLPV}.

\begin{thm}
\cite{mohammadpour2012control}
Let system~\eqref{eq:closedloop} be given. 
Let $\Theta$ be the set of all the admissible parameters $\theta$.
Let $\tilde{A}(\cdot)$ be Lipschitz continuous,
with a Lipschitz constant $L_A$, for each $\theta\in\Theta$.
Assume that, for any fixed $\bar{\theta}\in\Theta$,
any solution to the LTI system
$
\dot{x}(t)=\tilde{A}(\bar{\theta})x(t)
$
is such that there exist constants $m\geq 1$ and $\lambda>0$ such that a$
	\Vert x(t) \Vert \leq m e^{-\lambda t} \Vert x(0) \Vert$. 
	If $\Vert \dot{\theta}(t)  \Vert < \frac{{\lambda}^2}{4 L_A m \log(m)},\, \forall t \geq 0$,
then $0$ is exponentially stable, with respect to system~\eqref{eq:closedloop}.
\label{thm:stability}
\end{thm}

\begin{lem}
\label{lem:subdomain}
Let Assumption~\ref{ass:KKK} and Assumption~\ref{ass:exp} hold.
There exists $\mathcal{L}\subset\mathcal{D}$, such that,
if $x(0)\in\mathcal{L}$, then
$x(t)\in\mathcal{D}$, $\forall t \geq 0$.
\end{lem}

\begin{proof}
By Assumption~\ref{ass:KKK}, one has that \eqref{eq:conditions} holds.
Hence, by \cite{desoer1969slowly}, there exists a positive definite matrix $P(\theta)$, such that
\begin{subequations}
\begin{eqnarray}
-I & = & \tilde{A}(\theta)P(\theta)+\tilde{A}(\theta)^\top P(\theta),\label{eq:classlyap}\\
P(\theta) & = & \textstyle\int_{0}^{\infty} e^{\tau \tilde{A}(\theta)^\top}e^{\tau \tilde{A}(\theta)}d\tau,
\end{eqnarray}
\end{subequations}
where $I$ is the 2--dimensional identity matrix, for all $\theta\in\mathcal{E}$. By 
the proof of Proposition~\ref{prop:lipschi}, one has that the matrix $\tilde{A}(\theta)$
is differentiable. Hence, by \cite{wilcox1967exponential},
one has that $\frac{\partial e^{\tau \tilde{A}(\theta)}}{\partial \theta_i}=\int_0^\tau e^{(\tau-u) \tilde{A}(\theta)}
\frac{\partial \tilde{A}(\theta)}{\partial \theta_i}e^{u \tilde{A}(\theta)}du$, for $i=1,2$. 
Therefore, by item \ref{point:lip}) and item \ref{point:norm}) of Assumption~\ref{ass:exp},
one has that there exists constants $C_i>0$ and $\lambda_i>0$ such that
$\Vert \frac{\partial P(\theta)}{\partial \theta_i} \Vert \leq \int_{0}^{\infty} C_i\tau e^{-2\lambda_i\tau}d\tau<\infty$,
$i=1,2$.
Hence, there exists a constant $L_P$, such that $\Vert P(\theta)-P(\theta') \Vert < L_P \Vert \theta
-\theta' \Vert$, for each $\theta$, $\theta'\in\mathcal{E}$.
Hence, due to \eqref{eq:classlyap}, the function 
 $V=x^\top P(\theta)x$ is such that
$\dot{V}\leq x^\top(L_P \Vert \dot{\theta}\Vert  -1)x$, for $\dot{\theta}=[\begin{array}{cc}
\psi_1(x) & \psi_2(x) 
\end{array}]^\top$, $\forall x \in \mathcal{D}$.
Hence, by the continuity of the functions $\psi_1(x)$ and $\psi_2(x)$ and 
since $\psi_1(0)=\psi_2(0)=0$,
there exists a domain $\mathcal{H} = \{ x\in\mathcal{D}:\,\Vert [\begin{array}{cc}
\psi_1(x) & \psi_2(x)
\end{array}]^\top \Vert < L_P^{-1}   \}$
such that $0\in\mathcal{H}$. Thus, let $c>0$ be the largest constant
such that $\mathcal{W}(\theta)=\{x\in\mathbb{R}^2:\;x^\top P(\theta)x<c \}$
is a subset of $\mathcal{H}$, for all $\theta\in\mathcal{E}$.
Define the set $\mathcal{L}=\cap_{\theta\in\mathcal{E}}\mathcal{W}(\theta)$. 
Since $V$ is a Lyapunov function, with respect to system~\eqref{eq:closedloop},
one has that, if $x(0)\in\mathcal{L}$, then $x(t)\in\mathcal{H}\subset\mathcal{D}$,
for all times $t\geq 0$.
\end{proof}

\begin{thm}
\label{thm:conv}
Let the assumptions of Lemma~\ref{lem:subdomain} hold. 
Then $0$ is exponentially stable for system~\eqref{eq:closedloop},
with $\mathcal{L}$ being a conservative estimate of the attraction domain.
\end{thm}

\begin{proof}
By Lemma~\ref{lem:subdomain}, if $x(0)\in\mathcal{L}$, then $x(t)\in\mathcal{D}$, for all times $t\geq 0$.
Hence, since, if $x(t)\in\mathcal{D}$, for all times $t\geq 0$, all the assumptions of Theorem~\ref{thm:stability} hold,
then $0$ is exponentially stable, with respect to system~\eqref{eq:closedloop}.
\end{proof}

By Theorem~\ref{thm:conv}, 
 $K(\theta)$ and $\mathcal{L}$ are a solution to Problem~\ref{prob:SOFLPV}.

A gain $K$, dependent on the parameters $\theta$
 has been computed by solving \eqref{eq:conditions}, with $\lambda=15$,
 by using the procedure given in Section~\ref{sec:ineq} and 
 Algorithm~\ref{alg:solutionPUR} to compute the polynomials $\eta(\theta,t)$
 and $\varrho_K(\theta,t)$. 
One has that the domain $\mathcal{D}=\{x=[\begin{array}{cc}
\alpha & q
\end{array}]^\top\in\mathbb{R}^2:\;
\vert \alpha \vert \leq 100,\,\vert q \vert \leq 100 \}$ is such that all the conditions of
 Assumption~\ref{ass:exp} hold, with $L_A = 50$, ${\lambda}=14.5$ and $m=1.0026$.
 A simulation has been carried out, with $K(\theta)=\varrho_K(\theta,\bar{t})$,
 where $\bar{t}$ is a solution to $\eta(\theta,t)=0$, from 
 initial conditions starting inside the domain $\mathcal{Q}=\{x\in\mathcal{D}:\;\alpha\geq 0, \,q\geq -5,\,  
 q \leq 100 \} \cup \{x\in\mathcal{D}:\;\alpha \leq 0,\,q\geq-100,\,q\leq 5 \}$.
 Figure~\ref{fig:clostraj} shows the trajectories of the system~\eqref{eq:closedloop} with $x(0)\in\mathcal{Q}$.
 
 \begin{figure}[htb]
	 \centering
	 \includegraphics[width=0.43\textwidth]{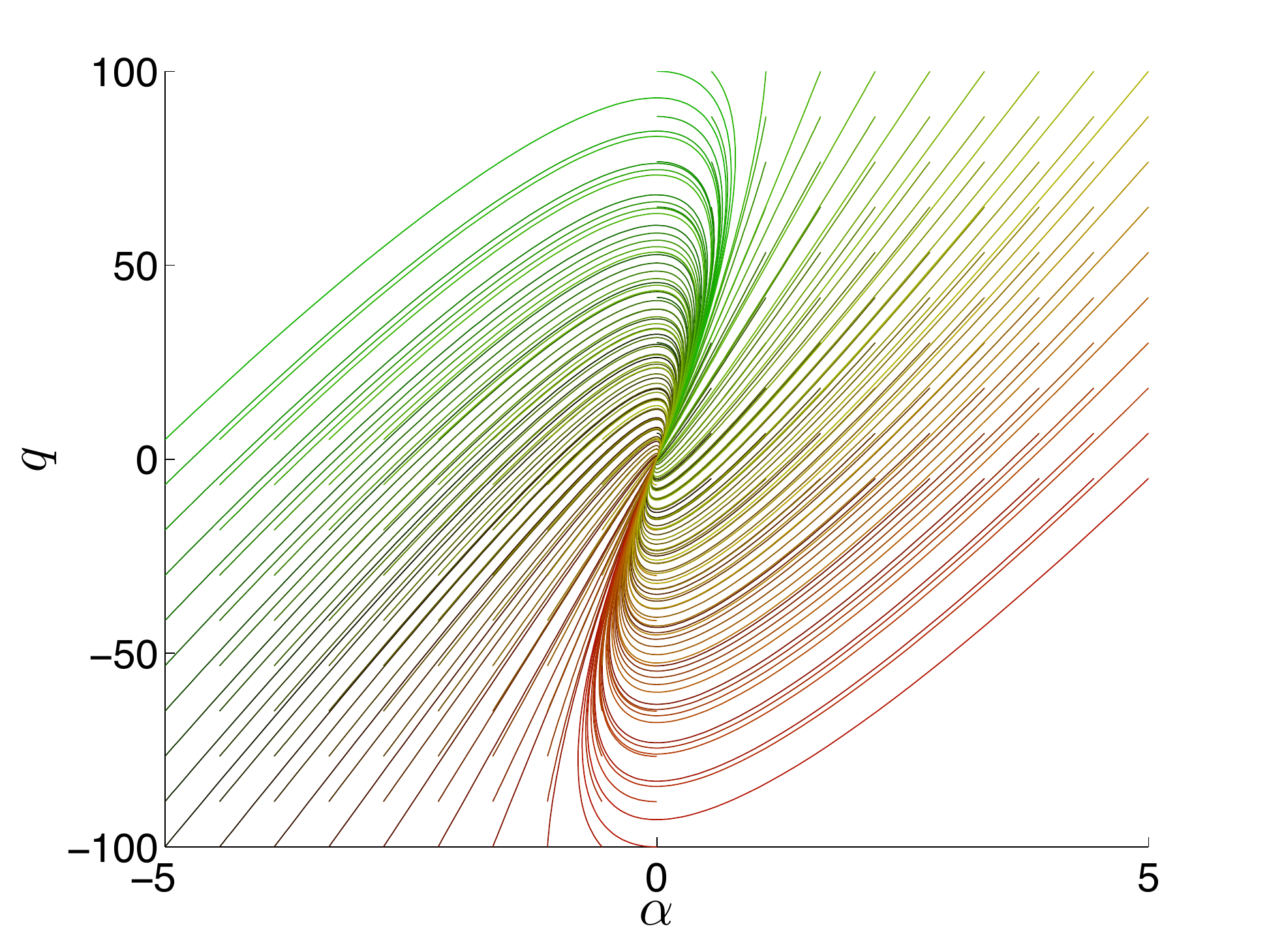}
	 \caption{Trajectories of system~\eqref{eq:closedloop}, with $x(0)\in\mathcal{Q}$.\label{fig:clostraj}}
 \end{figure}

\section{Conclusions}
Three algorithmic procedures for solving systems of polynomial inequalities are described. The first step,
common to the three, is the classical \cite{anderson1977output} reduction to  a system of equalities.
To solve this last system two available methods have been adapted and a new one has been derived.
The latter one has been used to solve the Static Output Feedback stabilization problem for a LPV system.

\bibliographystyle{ieeetr}
\bibliography{biblio}

\end{document}